\documentclass[12pt]{amsart}
\usepackage{amsthm,amsmath, amsfonts, amssymb, comment}
\newtheorem{theorem}{Theorem}[section]
\newtheorem{lemma}[theorem]{Lemma}
\newtheorem{proposition}[theorem]{Proposition}

\theoremstyle{definition}
\newtheorem{definition}[theorem]{Definition}

\theoremstyle{remark}

\usepackage{mathrsfs}
\usepackage[english]{babel}
\usepackage{graphicx}   
\usepackage{tikz}
\usepackage{subcaption}
\usepackage{url}        
\usepackage{bm}         
\usepackage{multirow}
\usepackage{booktabs}
\usepackage{algorithm}
\usepackage{algorithmic}
\usepackage{enumerate}

\newcommand{\aut}{{\mbox{\rm Aut}}}
\newcommand{\id}{{\mbox{\rm id}}}

\usepackage[colorlinks, linkcolor=black]{hyperref}

\providecommand*{\remarkautorefname{Remark}}
\hypersetup{colorlinks, bookmarks, unicode} 
\usepackage{multicol}
\usepackage[noadjust]{cite}
\usepackage{indentfirst}

\begin{document}

\title[Omnigenous Locally Finite Groups]{On the Isomorphism Relation for Omnigenous Locally Finite Groups}

\author{Su Gao}
\address{School of Mathematical Sciences and LPMC, Nankai University, Tianjin 300071, P.R. China}
\email{sgao@nankai.edu.cn}
\thanks{The authors acknowledge the partial support of their research by the Fundamental Research Funds for the Central Universities and by the National Natural Science Foundation of China (NSFC) grant 12271263.}

\author{Feng Li}
\address{School of Mathematical Sciences and LPMC, Nankai University, Tianjin 300071, P.R. China}
\email{fengli@mail.nankai.edu.cn}

\date{\today}

\begin{abstract} 
The concept of an omnigenous locally finite group was introduced in \cite{EGLMM} as a generalization of Hall's universal countable locally finite group. In this paper we show that the class of all countable omnigenous locally finite groups is Borel complete, hence it has the maximum Borel cardinality of isomorphism types among all countable structures.
\end{abstract}

\maketitle
\section{Introduction}    

In \cite{Hall}, Philip Hall proved the existence and the uniqueness, up to isomorphism, of a countable locally finite group $\mathbb{H}$ with the following properties:
\begin{enumerate}
\item[(i)] Every finite group can be embedded in $\mathbb{H}$;
\item[(ii)] For any finite groups $F, G$ and embeddings $\varphi\colon F\to \mathbb{H}$ and $i\colon F\to G$, there is an embedding $\Phi\colon G\to \mathbb{H}$ such that $\varphi=\Phi\circ i$.
\end{enumerate}
In modern model theory (see e.g. \cite{Ho}) property (i) is called {\em universality} (for finite groups) and property (ii) is called {\em weak homogeneity}. It follows from the weak homogeneity that $\mathbb{H}$ is also universal for all countable locally finite groups. For this reason, $\mathbb{H}$ is now called {\em Hall's universal countable locally finite group} (see e.g. \cite{EGLMM}). For readers who are familiar with the Fra\"{i}se\'{e} theory, $\mathbb{H}$ can also be regarded as the Fra\"{i}se\'{e} limit of the Fra\"{i}se\'{e} class of all finite groups. In any case, $\mathbb{H}$ turns out to be a rather canonical object.

In \cite{EGLMM} the following notion more general than the weak homogeneity was introduced. We call a locally finite group $\Gamma$ {\em omnigenous} if for any finite groups $F, G$ and embeddings $\varphi\colon F\to \Gamma$ and $i\colon F\to G$, there is a finite subgroup $H$ of $\Gamma$ and a surjective homomorphism $\Psi\colon H\to G$ such that $\Psi\circ \varphi=i$. Of course, $\mathbb{H}$ is omnigenous. It was shown in \cite{EGLMM} that, in contrast to the uniqueness of $\mathbb{H}$, there exist continuum many pairwise non-isomorphic countable omnigenous locally finite groups. Furthermore, each of these continuum many groups can be made to contain a copy of $\mathbb{H}$ as a subgroup, and therefore is universal for all countable locally finite groups. 

In this paper, we study the isomorphism relation for all countable omnigenous locally finite groups from the point of view of descriptive set theory. To prove our main result we use the framework of invariant descriptive set theory (see \cite{GaoBook}), particularly the notion of Borel reduction between equivalence relations on Polish spaces.



Friedman and Stanley \cite{FS} introduced and studied the notion of Borel reducibility between the isomorphism relations of Borel classes of countable structures. They discovered that among all isomorphism relations of Borel classes of countable structures there is a most complex one in terms of Borel reducibility, and they called such Borel classes Borel complete. More precisely, a Borel class of countable structures or the isomorphism relation of such a class is {\em Borel complete} if the isomorphism relation of any kind of countable structures is Borel reducible to it. In other words, Borel complete isomorphism relations have the maximum complexity in the Borel reducibility hierarchy among all isomorphism relations of countable structures. 

Friedman and Stanley \cite{FS} proved the Borel completeness for many Borel classes of countable structures. These include the classes of countable graphs, countable (rooted) trees, countable linear orders, and countably infinite fields of characteristic $p$ (where $p$ is zero or a prime). In particular, using results of Mekler \cite{Me} (or see \cite[A.3]{Ho}), they showed that the isomorphism relation of all countable groups is Borel complete (see \cite[(2.3)]{FS}). In fact, Mekler's result demonstrates that the isomorphism relation of all countable nilpotent groups of class 2 and of exponent $p$, where $p$ is any odd prime, is Borel complete. Furthermore, the groups constructed by Mekler are all locally finite, and therefore the isomorphism relation of all countable locally finite groups is Borel complete. 

After the work of Friedman and Stanley, more Borel classes of countable structures were studied and found to be Borel complete by various authors. These include the class of all countable Boolean algebras (Camerlo and Gao \cite{CG}), certain subclasses of countable graphs and other homogeneous countable structures (Clemens \cite{Cl}; Clemens, Coskey, and Potter \cite{CCP}), and countable models of set theory (Clemens, Coskey, and Dworetzky \cite{CCD}) and arithmetic (Coskey and Kossak \cite{CK}). Recently, Paolini and Shelah \cite{PS1} proved that the class of all countable torsion-free abelian groups is Borel complete. This settles a significant question in this area posed by Friedman and Stanley \cite{FS}. 

Meanwhile, the notion of Borel reducibility has been extended to the context of arbitrary equivalence relations on Polish spaces and of classification problems in mathematics. Becker and Kechris \cite{BK} showed that Borel complete isomorphism relations are exactly the most complex orbit equivalence relations induced by Borel actions of the infinite permutation group $S_\infty$ (therefore they are also called {\em $S_\infty$-universal} in the current literature). Many classification problems in mathematics were found to be Borel complete, meaning that these classification problems, as equivalence relations on Polish spaces, are Borel bireducible to a Borel complete isomorphism relation. These include the classification problems of following mathematical objects:
\begin{itemize}
\item[-] Zero-dimensional compact metrizable spaces up to homeomorphism (Camerlo and Gao \cite{CG});
\item[-] Zero-dimensional locally compact Polish metric spaces up to isometry (Gao and Kechris \cite{GK});
\item[-] Polish ultrametric spaces  up to isometry (Gao and Kechris \cite{GK});
\item[-] Dimension groups and (commutative) AF $C^*$-algebras up to isomorphism, and Bratteli diagrams up to equivalence (Camerlo and Gao \cite{CG});
\item[-] Homeomorphisms of the Cantor space (also known as Cantor systems in topological dynamics) up to topogical conjugacy (Camerlo and Gao \cite{CG});
\item[-] Automorphism groups of certain countable structures up to topological conjugacy (Beserra and Coskey \cite{BC}; Coskey and Ellis \cite{CE0}, \cite{CE1}, \cite{CE2}; Coskey, Ellis, and Schneider \cite{CES});
\item[-] Certain Polish metric spaces with a fixed distance set up to isometry (Camerlo, Marcone, and Motto Ros \cite{CMM});
\item[-] Locally compact Polish metric spaces up to isometry (Malicki \cite{Mal}).
\end{itemize}

The main theorem of this paper is the following.
\begin{theorem}\label{main theorem}
The isomorphism relation of countable universal, omnigenous locally finite groups is Borel complete. 
\end{theorem}

This result is somewhat surprising since the concepts of weak homogeneity and omnigenousness do not seem drastically different at first sight, but our result put the varieties of these countable locally finite groups at the two extremes of the Borel reducibility hierarchy. In the proof of our main theorem we use Mekler's construction as a starting point and reduce the class of all countable graphs to the class of groups under our consideration.



The rest of the paper is organized as follows. In Section~\ref*{sec:2} we give the basic definitions and some preliminary results. In Section \ref*{sec:3}, we present a method of construction for more omnigenous locally finite groups. Finally in Section \ref*{sec:4} we use the construction to prove our main theorem. 

{\em Acknowledgments}. We would like to thank Gianluca Paolini, David Schrittesser, and Mahmood Etedadialiabadi for useful discussions on this topic.

\section{Preliminaries}\label{sec:2}
In this section we review some basic concepts and notation, and give some preliminary results.

\subsection{Invariant descriptive set theory}
In this subsection, all the undefined concepts and basic results without reference can be found in \cite{GaoBook}. 

Recall that a topological space is {\em Polish} if it is separable and completely metrizable. The collection of all {\em Borel} sets in a Polish space is the smallest $\sigma$-algebra containing all open sets. If $X, Y$ are Polish spaces and $f\colon X\to Y$ is a map, then $f$ is {\em Borel} if for any open subset $U$ of $Y$, $f^{-1}(U)$ is a Borel subset of $X$. An equivalence relation $E$ on a Polish space $X$ is {\em Borel} if $E$ is a Borel subset of $X\times X$. 

Given equivalence relations $E, F$ on Polish spaces $X, Y$, respectively, we say that $E$ is {\em Borel reducible} to $F$, denoted $E\leq_B F$, if there is a Borel map $f\colon X\to Y$ such that for any $x_1, x_2\in X$, 
$$ x_1Ex_2\iff f(x_1)Ff(x_2). $$
We say that $E$ is {\em Borel bireducible} with $F$, denoted $E\sim_B F$, if both $E\leq_B F$ and $F\leq_B E$. We say that $E$ is {\em strictly Borel reducible} to $F$, denoted $E<_B F$, if $E\leq_B F$ but $F\not\leq_B E$. 

The quasi order $\leq_B$ gives a measurement of the relative complexity for equivalence relations. We also interpret the Borel reduction $E\leq_B F$ as saying that the {\em Borel cardinality} of the quotient space $X/E$ is at most that of $Y/F$. 

Invariant descriptive set theory studies the relative complexity of equivalence relations on Polish spaces via the notion of Borel reducibility. But before we can apply the concepts and results of invariant descriptive set theory, we need to gather our objects of interest together and form a Polish space. 

Let $\mathcal{L}$ be a countable language. If we denote by $\mathcal{C,R,F}$, respectively, the sets of all constant, relation and function symbols in $\mathcal{L}$, then we can view the space 
\begin{equation*}
    X_{\mathcal{L}}=\prod_{c\in \mathcal{C}} \omega \times \prod_{R\in \mathcal{R}}2^{\omega^{a(R)}}\times \prod_{f\in \mathcal{F}}\omega^{\omega^{a(f)}}
\end{equation*}
as the coding space of all countable $\mathcal{L}$-structures with universe $\omega$, where $a(\cdot)$ denotes the arity of a symbol. The isomorphism relation for all countably infinite $\mathcal{L}$-structures is then an equivalence relation on $X_\mathcal{L}$. If we equip $\omega$ and $2=\{0,1\}$ with discrete topologies and $X_{\mathcal{L}}$ with the product topology, then $X_{\mathcal{L}}$ becomes a Polish space. 

In this paper we will consider countable graphs and countable groups. For countable graphs the language consists of one binary relation symbol, and the space of all countable graphs with universe $\omega$ is 
$$\begin{array}{rl} \mathcal{P}=\big\{ x\in 2^{\omega^2}\colon & \forall n, m\in\omega\\
& x(n,n)=0 \mbox{ and } (x(n,m)=1\rightarrow x(m,n)=1)\big\}, 
\end{array}$$
which is a closed subspace of $2^{\omega\times\omega}$, and therefore a Polish space.

In the case of countable groups we use the language $\mathcal{L}=\{e, \cdot, {}^{-1}\}$, where $e$ is a constant symbol, $\cdot$ is a binary function symbol, and ${}^{-1}$ is a unary function symbol. For each $x\in X_{\mathcal{L}}$ let $\Gamma_x$ be the $\mathcal{L}$-structure coded by $x$. Let $\mathcal{G}$ be the subset of $x\in X_{\mathcal{L}}$ where $\Gamma_x$ is a group. Then $\mathcal{G}$ is a closed subset of $X_{\mathcal{L}}$ and therefore is itself a Polish space.

Recall that a group is {\em locally finite} if every finitely generated subgroup is finite. Consider the subsets
$$ \mathcal{G}_{\rm lf}=\{ x\in \mathcal{G}\colon \mbox{ $\Gamma_x$ is locally finite}\}, $$
$$ \mathcal{G}_{\rm omni}=\{ x\in \mathcal{G}\colon \mbox{ $\Gamma_x$ is omnigenous}\}, $$
and
$$ \mathcal{G}_{\rm univ}=\{x\in \mathcal{G}\colon \mbox{ $\Gamma_x$ is universal for all finite groups}\}. $$
Then each of these sets is a $G_\delta$ subset of $\mathcal{G}$, and therefore is a Polish space. Moreover, each of them is a set invariant for the isomorphism relation. The primary class of groups we study in this paper is
$$ \mathcal{G}_{\rm uolf}=\mathcal{G}_{\rm lf}\cap \mathcal{G}_{\rm omni}\cap \mathcal{G}_{\rm univ}, $$
which is again an invariant Polish subspace of $\mathcal{G}$. We are interested in the complexity of the isomorphism relation on $\mathcal{G}_{\rm uolf}$.

\subsection{Omnigenous locally finite groups}

We recall some more detailed properties of Hall's universal locally finite group $\mathbb{H}$. 

Hall \cite{Hall} considered the following strong form of {\em ultrahomogeneity} for a group $\Gamma$: Any two isomorphic finite subgroups of $\Gamma$ are conjugate by an element of $\Gamma$. He showed that $\mathbb{H}$ is the unique, up to isomorphism, countable locally finite group that is universal for finite groups and is ultrahomogeneous. 

For any group $G$ we denote the {\em identity} element or the {\it trivial} element by $e_G$. An element $g\in G$ is {\em nontrivial} if $g\neq e_G$. A group $G$ is {\em trivial} if $G=\{e_G\}$; a subgroup $H$ of $G$ is {\em trivial} if $H=\{e_G\}$. Recall that a group $G$ is {\em simple} if its only normal subgroups are $\{e_G\}$ and $G$. Hall \cite{Hall} showed that $\mathbb{H}$ is simple, a fact we will use in our proofs later. 





Although we will not need it in our proofs, we present a new characterization of omnigenous locally finite groups below. We first recall some basic group-theoretic notions. Recall that the \textit{index} of a subgroup $H$ of a group $G$ is the number of cosets of $H$ in $G$, or equivalently, the number of classes of the equivalence relation defined on $G$ by $$g\sim h \iff g^{-1}h\in H.$$ A group $G$ is {\em residually finite} if for any nontrivial $g\in G$ there is a normal subgroup $N \unlhd\, G$ of finite index such that $g\not\in N$. Finite groups and finitely generated free groups are residually finite (see \cite[Section 2]{Ma}).
We also recall the definition of the free product of groups. Given groups $G$ and $H$, a \textit{word} is a finite sequence of nontrivial elements of $G$ and $H$ (the sequence could be empty) and a word $w$ is \textit{reduced} if either it is an empty word or $w$ is nonempty and there are no adjacent elements in $w$ which are from the same group. If $w$ is a word, the reduction of $w$, denoted by $w^*$, is the unique reduced word obtained by successively replacing any occurence of consecutive elements from the same group by their product, if the product is nontrivial, and eliminating them if the product is trivial. The \textit{free product} of $G$ and $H$, denoted by $G*H$, is the group whose elements are the reduced words of $G$ and $H$, and the identity element is the empty word, with the group operation defined by $w_1\cdot w_2=(w_1w_2)^*$, where $w_1w_2$ is the concatenation of $w_1$ and $w_2$. 

\begin{lemma}\label{lem:wo} 
    Let $F, G$ be finite groups and let $\psi\colon F\to G$ be a homomorphism from $F$ into $G$. Then there exist a finite group $H$, an embedding $i\colon F\to H$, and a surjective homomorphism $\Psi\colon H\to G$ such that $\psi=\Psi\circ i$.
\end{lemma}

\begin{proof} 
    Let $g_1, \dots, g_n\in G\setminus \psi(F)$ be such that $G=\langle \psi(F), g_1, \dots, g_n\rangle$. Consider the free group $F_n=\langle x_1, \dots, x_n\rangle$. Define $\phi: F*F_n\to G$ by $\phi(f)=\psi(f)$ for $f\in F$ and $\phi(x_i)=g_i$ for $1\leq i\leq n$. Then $\phi$ can be extended to a surjective group homomorphism from $F*F_n$ to $G$. Note that $\ker \phi$ is of finite index in $F*F_n$, since $G$ is finite.

    By a classical theorem of Gruenberg (see \cite[Theorem 4.1]{Gruenberg} or \cite{Berlai}), the free product of residually finite groups is residually finite. Thus $\Gamma=F*F_n$ is residually finite. Note that we can naturally view $F$ as a subgroup of $F*F_n$. Thus for each nontrivial $g\in F$, there is a normal subgroup $N_g$ of finite index in $F*F_n$ with $g\notin N_g$. Let $N=\bigcap_{g\in F} N_g \cap \ker\phi$. Then $N\cap F=\{e_{\Gamma}\}$ and $N\leq \ker\phi$. Since $F$ is finite, $N$ is of finite index. Let $H=(F*F_n)/N$. Define $i: F\to H$ by $i(f)=fN$. Then $i$ is a homomorphism. To see that $i$ is an embedding, suppose $fN=f'N$ with $f, f'\in F$. Then $f^{-1}f'\in N\cap F=\{e_\Gamma\}$, so $f=f'$.

    Define $\Psi\colon H\to G$ by $\Psi(hN)=\phi(h)$. We note that $\Psi$ is well-defined. For this, suppose $hN=h'N$. Then $h^{-1}h\in N\leq \ker\phi$. Thus $\phi(h^{-1}h')=e_G$, and $\phi(h)=\phi(h')$. It is obvious that $\Psi$ is a surjective homomorphism.

    Finally, we verify that $\psi=\Psi\circ i$. For this, let $f\in F$, and note $\Psi\circ i(f)=\Psi(fN)=\phi(f)=\psi(f)$.
\end{proof}

\begin{proposition} Let $\Gamma$ be a locally finite group. Then $\Gamma$ is omnigenous iff for any finite subgroup $F\leq \Gamma$, finite group $G$, and embedding $\psi\colon  F\to G$, assuming that there is $g\in G$ such that $G=\langle \psi(F), g\rangle$, then there is a finite $H\leq \Gamma$ with $F\leq H$ and a surjective homomorphism $\Psi\colon H\to G$ such that $\Psi\upharpoonright F=\psi$.
\end{proposition}

\begin{proof} The stated property is obviously necessary for omnigenousness. To show its sufficiency, let $\Gamma$ have the property. 

Suppose $F\leq \Gamma$ is finite, $G$ is a finite group, and $\psi\colon F\to G$ is an embedding. Suppose $n\geq 1$ is the smallest number such that there are $g_1, \dots, g_n\in G$ such that
$G=\langle \psi(F), g_1, \dots, g_n\rangle$. We prove by induction on $n$ that there exist a finite $H\leq \Gamma$ with $F\leq H$ and a surjective homomorphism $\Psi\colon H\to G$ with $\Psi\upharpoonright F=\psi$. If $n=1$ this is given by the assumption on $\Gamma$. 

As the inductive hypothesis, suppose there exist a finite $H_0\leq \Gamma$ with $F\leq H_0$ and a surjective homomorphism $\Psi_0\colon H_0\to \langle \psi(F), g_1, \dots, g_{n-1}\rangle$ with $\Psi_0\upharpoonright F=\psi$. By Lemma~\ref{lem:wo}, there exist a finite group $K$, an embedding $i\colon H_0\to K$, and a surjective homomorphism $\kappa\colon K \to G$ such that $\Psi_0=\kappa\circ i$. Note that, without loss of generality, we may assume $K=\langle i(H_0), k\rangle$ for any $k$ with $\kappa(k)=g_n$. By the assumption on $\Gamma$ again, there exist $H\leq \Gamma$ with $H_0\leq H$ and a surjective homomorphism $\Phi\colon H\to K$ such that $\Phi\upharpoonright H_0=i$. Let $\Psi=\kappa\circ \Phi$. Then $\Psi\colon H\to G$ is surjective, and $\Psi\upharpoonright F=\Psi_0\upharpoonright F=\psi$.
\end{proof}

\section{More Omnigenouis Locally Finite Groups}\label{sec:3}
In this section, we give a method to construct pairwise non-isomorphic countable omnigenous locally finite groups. 

\subsection{A general construction}

We first consider direct limits of direct systems. By a {\em direct system} we mean a sequence $\{G_i, \varphi_i\colon i\in\omega\}$ where for each $i\in\omega$, $G_i$ is a group and $\varphi_i\colon G_i\to G_{i+1}$ is an embedding. The {\em direct limit} of a direct system is defined as follows. First, define a relation $\sim$ on the disjoint union $X=\bigsqcup_{i\in\omega} G_i$. If $a \in G_i$, $b\in G_j$ and $i<j$, then define
$$a\sim b  \iff b\sim a \iff (\varphi_{j-1}\circ \dots\circ\varphi_i)(a)=b.$$ 
Since each $\varphi_i$ is an embedding, $\sim$ is an equivalence relation. The direct limit will be a group on the set of all $\sim$ classes, i.e., $\{[a]_\sim\colon a\in X\}$. 
To define the group operation, consider $a\in G_i$ and $b\in G_j$ and suppose $i\leq j$. Define 
$$[a]_\sim \cdot [b]_\sim = [(\varphi_{j-1}\circ \dots\circ \varphi_i)(a)\cdot b]_\sim.$$ It is routine to check that $\cdot$ is welldefined and that $[a]_\sim^{-1}=[a^{-1}]_\sim$. In the rest of this paper, when there is no danger of confusion, we will view each embedding as inclusion. Thus for all $i$, $G_i\leq G_{i+1}$ and the direct limit is just the union of all $G_i$. For $a\in G_i$, we will also simply denote $[a]_\sim$ as $a$ and view it as an element of each $G_j$ when $j\geq i$. 

Now let $A, H$ be groups with $H$ being a non-abelian simple group. Then $H$ has trivial center. Let $f:A\oplus H \to H$ be a group embedding. Define a group embedding $\varphi: A\oplus H \to A\oplus H$ by 
$$\varphi(s,t)=(s,f(s,t)) $$
for $s\in A$ and $t\in H$. Consider a direct system $\{G_i,\varphi_i\colon i\in\omega\}$ where each $G_i$ is an isomorphic copy of $A\oplus H$ and $\varphi_i$ is the above $\varphi$. Let $D(A,H,f)$ be the direct limit of this system. We will always use the same $H$ and so will just denote it by $D(A,f)$. 

There is a special subgroup of $D(A,f)$ which we denote by $S(A,f)$. Consider the direct system $\{H_i, \psi_i\colon i\in\omega\}$ where for each $i\in\omega$, 
$$H_i= H\leq G_i \mbox{ and }\psi_i=\varphi_i\upharpoonright H_i.$$
Let $S(A,f)$ be the direct limit of this system. Then $S(A,f)$ is special in the following sense. 

\begin{lemma}\label{special subgroup}
    $S(A,f)$ is a simple normal subgroup of $D(A,f)$. Every nontrivial normal subgroup of $D(A,f)$ contains $S(A,f)$ as a normal subgroup. Thus $S(A,f)$ is the only nontrivial simple normal subgroup of $D(A,f)$. 
\end{lemma}

\begin{proof}
    It is easy to see that $S(A,f)$ is a normal subgroup. To see that it is simple, it suffices to show that the normal subgroup generated by any nontrivial element of $S(A,f)$ must be the whole group. For this, fix a nontrivial $t \in H_i$. Let $N_t$ be the normal subgroup generated by $t$ in $S(A,f)$. Then $N_t \cap H_j \ni t$ is a nontrivial normal subgroup of $H_j$ for all $j\geq i$. Since we have assumed that $H_j=H$ is simple, $N_t \cap H_j =H_j$ and thus $N_t = S(A,f)$. This shows that $S(A,f)$ is simple. 

    Now assume $G$ is a nontrivial normal subgroup of $D(A,f)$. Let $(e_A,e_H)\neq (s,t)\in G\cap G_i$. We may assume that $t\neq e_H$ since otherwise we can do the same argument to $\varphi_i(s,t)=(s,f(s,t)) \in G\cap G_{i+1}$. The group $H$ has trivial center. Thus we can choose $t'\in H_i$ such that $t$ and $t'$ do not commute. Then $(s,t)\neq (s,t't{t'}^{-1})=(e_A,t')(s,t)(e_A,t'^{-1})\in G\cap G_i$. Let $a=t'tt'^{-1}t^{-1}$. Then $a\in H_i$ is nontrivial and $(e_A,a)\in G\cap H_i$ (identify $H_i$ with $\{e_A\}\times H_i$). What we have shown so far is that if $G\cap G_i$ is nontrivial, then for all $j>i$, $G\cap H_j$ is nontrivial. Now by the argument in the preceding paragraph, the normal subgroup of $D(A, f)$ generated by any nontrivial element of $H_j$ is $S(A, f)$. Since $G$ is normal, it follows that $G$ contains $S(A, f)$.


    The last assertion  follows immediately. 
\end{proof}

We note that it is possible to characterize the isomorphism type of $A$ from that of $D(A,f)$. 

\begin{lemma}\label{back to A} For any group $A$ we have $A \cong D(A,f)/S(A, f)$. Consequently, if $D(A,f)\cong D(B,g)$ (it is not necessary that the group $H$ used in their constructions are the same), then $A\cong B$.  
\end{lemma}

\begin{proof}
    For any $g=(s,t)\in G_i=A\oplus H$, define $\pi(g)=s$. Then $\pi$ is well defined since $\varphi_i$ does not change the first coordinate. $\pi$ is obviously a homomorphism from $D(A,f)$ onto $A$. 
    Note that $\pi(g)=e_A$ iff $ g=(e_A,t)\in G_i=A\oplus H$ for some $i\in\omega$ and $t\in H$ iff $g\in S(A,f)$. Thus the kernel of $\pi$ is $S(A,f)$. By the isomorphism theorem, $A\cong D(A,f)/S(A,f)$.  
\end{proof}

\subsection{Some properties of the construction}
In this subsection we give more details and properties of the groups that will be used in the proof of Theorem \ref*{main theorem}.

\begin{proposition}\label{omnigeneity}
    Let $A$ be a locally finite group, let $H=\mathbb{H}$, and let $f\colon A\oplus \mathbb{H}\to \mathbb{H}$ be an embedding. Then $D(A,f)=D(A, H, f)$ is omnigenous. 
\end{proposition}

\begin{proof}
    Let $\Gamma_1\leq \Gamma_2$ be finite groups and $g:\Gamma_1\to D(A,f)$ be a group embedding. Since $\Gamma_1$ is finite, there is $i\in\omega$ such that $g(\Gamma_1)\subseteq G_i= A\oplus \mathbb{H}$. We should emphasize here that in $D(A,f)$, $g\in G_i$ is identified with $\varphi_i(g) \in G_{i+1}$. Then $f\circ g$ is an embedding from $\Gamma_1$ to $\mathbb{H}$. 
    By the weak homogeneity of $\mathbb{H}$, there exist a finite subgroup $K$ of $\mathbb{H}$ and a group isomorphism $h\colon K \to \Gamma_2$ so that $(f\circ g)(\Gamma_1)\leq K\leq \mathbb{H}$ and for all $\gamma\in \Gamma_1$, $(h\circ f\circ g)(\gamma)=\gamma$. Let $K'$ be the subgroup of $G_{i+1}$ generated by $(\varphi_i\circ g)(\Gamma_1)\cup (\{e_A\}\times K)$. Then since in $D(A,f)$, $g(\Gamma_1)$ is identified with $(\varphi_i\circ g)(\Gamma_1)$, $g(\Gamma_1)$ is a subgroup of $K'$. Since $A$ is locally finite and $K'$ is finitely generated, $K'$ is finite. 
    Let $\Phi$ canonically map $(s,t)\in G_{i+1}$ to $t\in \mathbb{H}$. Then it is easy to verify that $\Phi(K')=K$ and thus $h\circ \Phi$ is a homomorphism from $K'$ onto $\Gamma_2$. Now for any $\gamma\in \Gamma_1$, we have
    $$(h\circ \Phi\circ g)(\gamma)=(h\circ \Phi\circ \varphi_i\circ g)(\gamma)=(h\circ f\circ g)(\gamma)=\gamma.$$ Thus $D(A,f)$ is omnigenous. 
\end{proof}

From now on, we take all $H$ in our constructions to be Hall's group $\mathbb{H}$. We will use the following theorem of Paolini and Shelah \cite[Theorem 4]{PS} regarding the automorphism group of $\mathbb{H}$. 

\begin{proposition}[Paolini--Shelah \cite{PS}]\label{PSlemma}
     For every countable locally finite group $G$ there exists $G \cong G' \leq \mathbb{H}$ such that every $\phi \in \aut(G')$ extends to a $\hat{\phi} \in \aut(\mathbb{H})$ in such a way that $\phi \mapsto \hat{\phi}$ embeds $\aut(G')$ into $\aut(\mathbb{H})$.
\end{proposition}





What we will do is basically the following. Let $G$ be a locally finite group and let $A\leq G$ be a subgroup. We will choose a group embedding $f$ from $G\oplus\mathbb{H}$ into $\mathbb{H}$ by Proposition \ref{PSlemma} so that every automorphism of $f(G\oplus \mathbb{H})$ can be extended to an automorphism of $\mathbb{H}$. With this embedding, we can do the construction to form $D(G,f)$. Then $f$ will be naturally restricted on $A\oplus \mathbb{H}$ to form $D(A,f)=D(A,f\upharpoonright A\oplus \mathbb{H})$. Then $D(A,f)$ will be a subgroup of $D(G,f)$. 
The special property of $f$ helps us keep the isomorphism type of subgroup of $G$ in the following way. 

\begin{lemma}\label{automorphism}
    Let $G,f$ be as above. Let $A,B\leq G$. Suppose there is an automorphism of $G$ mapping $A$ onto $B$. Then there is an automorphism of $D(G,f)$ mapping $D(A,f)$ onto $D(B,f)$. 
\end{lemma}

\begin{proof} Let $\pi$ be an automorphism of $G$ mapping $A$ onto $B$.
    First of all, $\pi$ canonically extends to an automorphism of $G\oplus \mathbb{H}$ which we still denote by $\pi$, that is, $\pi(s,t)=(\pi(s),t)$. By the choice of $f$ which uses Proposition~\ref*{PSlemma}, there is an automorphism $\pi'$ of $\mathbb{H}$ such that for all $(s,t)\in G\oplus \mathbb{H}$, $(\pi'\circ f)(s,t)=(f\circ \pi)(s,t)$. Using this, we now define a sequence of automorphisms $h_i$ of $\mathbb{H}$ by induction on $i\in\omega$. 

    Let $h_0=\id$. Suppose $h_i$ has been defined. Then $\pi'\circ f\circ(\id,h_i)\circ f^{-1}$ is an automorphism of $f(G\oplus \mathbb{H})$ and thus extends to an automorphism $h_{i+1}$ of $\mathbb{H}$. This finishes the definition of $h_i$, $i\in\omega$. 
    
We now define an automorphism $\hat{\pi}$ of $D(G,f)$ mapping $D(A,f)$ onto $D(B,f)$. For $x\in D(G,f)$, if $x\in G_i=G\oplus\mathbb{H}$ is of the form $(s,t)$, we let $\hat{\pi}(x)=(\pi(s),h_i(t))\in G_i$. 
    We first check that $\hat{\pi}$ is well defined. For this, suppose $(s,t)\in G_i$ and $(s',t')\in G_{i+1}$ represent the same element of $D(G,f)$. Then $s'=s$ and $t'=f(s,t)$. Then 
    \begin{align*}
        h_{i+1}(t')&=(h_{i+1}\circ f)(s,t)=(\pi'\circ f)(s,h_i(t))\\& =(f\circ\pi)(s,h_i(t))=f(\pi(s),h_i(t)).
    \end{align*}
    Thus $(\pi(s),h_i(t))$ and $(\pi(s'),h_{i+1}(t'))$ represent the same element. 

   It is easy to check that $\hat{\pi}$ is an automorphism of $D(G, f)$. Finally, note that 
    \begin{align*}
        x\in D(A,f) &\iff \exists i\in\omega\  \big(\, x=(s,t)\in G_i \mbox{ and } s \in A\,\big) \\& \iff \exists i\in\omega \ \big(\,\hat{\pi}(x)=(\pi(s),h_i(t))\in G_i \mbox{ and } \pi(s)\in B\,\big)\\& \iff \hat{\pi}(x)\in D(B,f).  
    \end{align*}
    Thus $\hat{\pi}(D(A, f))=D(B, f)$ as promised.
\end{proof}

\section{Proof of the Main Theorem}\label{sec:4}

In this section, we prove the Borel completeness of the class of universal, omnigenous locally finite groups. 

We will use the construction of Mekler \cite{Me}; for a clearer description of Mekler's construction, see A.3 and Exercise 9.2.23 of \cite{Ho}. Mekler associates to each countable graph $G$ a countable group $\Gamma(G)$ which is nilpotent of class $2$ and of exponent $p$, where $p$ is a fixed odd prime. Since every solvable periodic group is locally finite (see e.g. \cite[Proposition 1.1.5]{Dixon}), each $\Gamma(G)$ is locally finite. 
The association map $\Gamma$ sends isomorphic graphs to isomorphic groups, and it turns out that $\Gamma$ is a Borel map from $\mathcal{P}$ to $\mathcal{G}_{\scriptsize\rm lf}$. Moreover, he uses a technical restriction on the graphs so that he can recover the graph $G$ from the group $\Gamma(G)$. 

\begin{definition}
    A graph is called \textit{nice} if it does not contain a triangle or a square, and for any distinct vertices $x,y$, there is a vertex $z$ adjacent to $x$ but not to $y$. 
\end{definition}

The class of all countable nice graphs is easily seen to be a $G_\delta$ subspace of $\mathcal{P}$. The following theorem paraphrases Mekler's results in \cite{Me}.

\begin{theorem}[Mekler \cite{Me}]\label{me}{\ }
\begin{enumerate}
\item[\rm (1)] The class of all countable nice graphs is Borel complete.
\item[\rm (2)] If $G$, $H$ are nice graphs, then $G\cong H$ iff $\Gamma(G)\cong \Gamma(H)$. 
\end{enumerate}
\end{theorem}

Thus it follows from Theorem~\ref{me} that the class of all countable locally finite groups is Borel complete. 

Next, we recall the following equivalent definition of the {\em random graph} or the {\em Rado graph} (see e.g. \cite[Theorem 7.4.4]{Ho}).

\begin{definition}\label{def:rg} 
    The {\em random graph} $\mathcal{R}$ is the unique countable graph  with the following property: 
\begin{quote}
   for any disjoint finite subsets $A,B$ of $\mathcal{R}$, there is a vertex $x$ in $\mathcal{R}$ which is adjacent to every vertex in $A$ but not to any in $B$. 
\end{quote}
\end{definition}

Now let $G$ be a countable graph. We define $G'$ to be an extension of $G$. The vertex set of $G'$ is the disjoint union of $G$ and $F^G$, where $F^G$ consists of all finite subsets of $G$. For $x,y\in G'$, we then define $xG'y$ if both of them are in $G$ and $xGy$ or exactly one of them, say $y$, is in $F^G$ and $x\in y$. Given a countable graph $G$, we define 
$G_0=G$, $G_{n+1}=G_n'$ and $\mathcal{R}(G)=\bigcup_n G_n$. 

\begin{lemma}
    For any countable graph $G$, we have $\mathcal{R}(G)\cong \mathcal{R}$. Moreover, if $H$ is also a countable graph and $f\colon G\to H$ is a graph isomorphism, then $f$ extends to an isomorphism of $\mathcal{R}(G)$ to $\mathcal{R}(H)$. 
\end{lemma}

\begin{proof}
    For the first part, we only need to check that $\mathcal{R}(G)$ satisfies the property in Definition~\ref{def:rg}. For this, let $A, B$ be disjoint finite subsets of $\mathcal{R}(G)$. Then there is $n\in\omega$ such that $A,B\subseteq G_n$ and thus $A\in F^{G_n}\subseteq G_{n+1}$. As a vertex in $\mathcal{R}(G)$, $A$ is adjacent to each element of $A$ but not to any of $B$. This shows that $\mathcal{R}(G)\cong \mathcal{R}$.

For the second part, note that if we can construct graph isomorphisms $f_n\colon G_n\to H_n$ such that for all $n\in\omega$, $f_{n+1}$ extends $f_n$, then $\bigcup_n f_n$ is an isomorphism of $\mathcal{R}(G)$ to $\mathcal{R}(H)$. 
    So it suffices to show that any isomorphism $f:G\to H$ can be extended to an isomorphism from $G'$ to $H'$. For this, define $F\colon G'=G\sqcup F^G\to H'=H\sqcup F^H$ by
$$ F(x)=\left\{\begin{array}{ll}
f(x), & \mbox{if $x\in G$,} \\
\{f(y)\colon y\in x\}, & \mbox{if $x\in F^G$.}
\end{array}\right.
$$
It is easy to verify that $F$ is an isomorphism extending $f$. 
\end{proof}

Thus we can view each graph $G$ as an induced subgraph of $\mathcal{R}$. Moreover, by carefully arranging the construction, we can find a copy of every graph $G$ in $\mathcal{R}$ as an induced subgraph in a Borel way. Let $\mbox{SG}(\mathcal{R})$ be the space of all subgraphs of $\mathcal{R}$. $\mbox{SG}(\mathcal{R})$ can be viewed as a closed subspace of $2^{\mathcal{R}}$, and is therefore a Polish space. To summarize the above extension procedure, we have the following theorem.

\begin{theorem}\label{graph extension}
    There is a Borel map $S$ from $\mathcal{P}$ to $\mbox{\rm SG}(\mathcal{R})$ such that for every graph $G\in\mathcal{P}$, $S(G)\cong G$, and for any $G, H\in\mathcal{P}$ and any isomorphism $f\colon S(G)\to S(H)$, there is an automorphism of $\mathcal{R}$ extending $f$. 
\end{theorem}
To continue, we should review the construction of Mekler. Fix an odd prime and a graph $G$. Define $\Gamma(G)$ as follows. First take the vertices of $G$ as group generators and form the free nilpotent-class-2, exponent-$p$ group, which is denoted by $\mathbf{2}_{x\in G} C_p$, where $C_p$ is the cyclic group of order $p$. More precisely, let $*_{x\in G} C_p$ be the free product of $C_p$'s generated by vetices of $G$. Then $\mathbf{2}_{x\in G}C_p=*_{x\in G}C_p/\langle [x,[y,z]]\colon x,y,z\in G \rangle$ where for group elements $x,y$, $[x,y]=x^{-1}y^{-1}xy$ is their commutator. Then $\Gamma(G)$ will be $\mathbf{2}_{x\in G}C_p/\langle [x,y]\colon xGy\rangle$. 

Our construction will use this idea uniformly. We first take the random graph and do the construction to get $\Gamma(\mathcal{R})$. Then $\Gamma(\mathcal{R})$ is generated by the vertex set of $\mathcal{R}$. For each countable graph $G$, let $\Gamma'(G)$ be the subgroup of $\Gamma(\mathcal{R})$ generated by $S(G)$. Then $\Gamma'$ is clearly Borel. 
\begin{lemma}\label{graph to group}
    For each countable graph, $\Gamma(G)$ is isomorphic to $\Gamma'(G)$. Moreover, for countable nice graphs $G$ and $H$, $G\cong H$ iff there is an automorphism of $\Gamma(\mathcal{R})$ sending $\Gamma'(G)$ onto $\Gamma'(H)$. 
\end{lemma}

\begin{proof}
    Recall that $\Gamma(G)$ is generated by $G$. Without loss of generality we can assume $G=S(G)\subseteq \mathcal{R}$. We prove that the embedding of generators of $\Gamma(G)$ to generators of $\Gamma(\mathcal{R})$ induces an embedding from $\Gamma(G)$ onto $\Gamma'(G)$. First note that since $\mathbf{2}_{x\in \mathcal{R}}C_p$ and $\mathbf{2}_{x\in G} C_p$ are both free among groups of nilpotent class 2 and exponent $p$, the embedding of generators naturally induces an embedding from the latter into the former. Thus we may assume that the latter is contained in the former. 
    Now let $N_G=\langle[x,y]\colon x,y\in G, xGy \rangle$ and $N_{\mathcal{R}}=\langle[x,y]\colon x,y\in \mathcal{R}, x\mathcal{R}y \rangle$. We define $f:\Gamma(G) \to \Gamma'(G)$ as follows. For $g\in \mathbf{2}_{x\in G}C_p$, we set $f(gN_G)=gN_\mathcal{R}$. Then it suffices to show that $f$ is a well-defined group embedding. 

    For the well-definedness, let $g,h\in \mathbf{2}_{x\in G}C_p$ and suppose that $h^{-1}g\in N_G$. Then since $G$ is a subgraph of $\mathcal{R}$, $h^{-1}g\in N_\mathcal{R}$. Thus $f$ is well-defined. Obviously, $f$ is a homomorphism. Thus we only need to show that $f$ is injective. For this let $g\in \mathbf{2}_{x\in G}C_p$ and suppose that $f(gN_G)=gN_\mathcal{R}=N_\mathcal{R}$, i.e., $g\in N_\mathcal{R}$. We need to show that $g\in N_G$. Here we need a property of the free nilpontent-class-2, exponent-$p$ group. As remarked in \cite{Me}, the center of $\mathbf{2}_{x\in G}C_p$ is a free abelian $p$-group generated by $\{[x,y]\colon x\neq y \in G\}$. Similarly for $\mathbf{2}_{x\in \mathcal{R}}C_p$. 
    Thus if $g\in N_\mathcal{R}$, then $g$ is in the center of $\mathbf{2}_{x\in \mathcal{R}}C_p$. Thus it is also in the center of $\mathbf{2}_{x\in G}C_p$ which means that there are $x_i,y_i\in G$ and $0<k_i<p$ such that $g=[x_0,y_0]^{k_0}\dots[x_n,y_n]^{k_n}$. Also since $g\in N_\mathcal{R}$, there are $x_i',y_i'\in\mathcal{R}$ such that $x_i'\mathcal{R}y_i'$ and $0<k_i'<p$ such that $g=[x_0',y_0']^{k_i'}\dots[x_m',y_m']^{k_i'}$. But then these two representations must be the same due to the freeness. Thus $x_iGy_i$ for all $i$ and $g\in N_G$. 

    For the moreover part, if $G\cong H$, then by Theorem \ref{graph extension}, there is an automorphism of $\mathcal{R}$ sending $S(G)$ onto $S(H)$. Then this automorphism induces an automorphism of $\Gamma(\mathcal{R})$ sending $\Gamma'(G)$ onto $\Gamma'(H)$. Conversely, if there is such an automorphism, then $\Gamma(G)\cong \Gamma'(G)\cong \Gamma'(H)\cong \Gamma(H)$. Thus by Theorem \ref{me} (2), $G\cong H$. 
\end{proof}

Finally we are ready to prove the Borel completeness of the class of all countable universal, omnigenous locally finite groups. First let $f:\Gamma(\mathcal{R})\oplus \mathbb{H}\to \mathbb{H}$ be an embedding such that every automorphism of the image subgroup extends to an automorphism of $\mathbb{H}$. Such an embedding exists by Proposition \ref{PSlemma}. Then let $L(\mathcal{R})=D(\Gamma(\mathcal{R}),f)$. For each nice graph $G$, let $L(G)=D(\Gamma'(G),f)\leq L(\mathcal{R})$. Then for nontrivial element $g\in L(\mathcal{R})$, $g\in L(G)$ iff $\pi(g)\in \Gamma'(G)$, where $\pi$ is as in the proof of Lemma \ref{back to A}. Thus the map $L$ is Borel. We then have: 
\begin{theorem}
    If $G$, $H$ are nice graphs, then the following are equivalent: 
    \begin{enumerate}
        \item[\rm (1)] $G\cong H$;  
        \item[\rm (2)] There is an automorphism of $L(\mathcal{R})$ mapping $L(G)$ onto $L(H)$; 
        \item[\rm (3)] $L(G)\cong L(H)$. 
    \end{enumerate}
\end{theorem}
\begin{proof}
    (1)$\Rightarrow$(2) follows from Lemmas \ref{graph to group} and \ref{automorphism}. (2)$\Rightarrow$(3) is obvious. (3)$\Rightarrow$(1) follows from Lemma \ref{back to A} and Theorem~\ref{me}. 
\end{proof}
By Proposition \ref{omnigeneity}, each $L(G)$ is omnigenous. It is also universal for finite groups since it contains a copy of $\mathbb{H}$. The proof of Theorem~\ref{main theorem} is thus complete.
\bibliographystyle{amsplain}

\end{document}